\documentclass[12pt]{article}

\usepackage{amsmath,amsfonts,amsthm,amssymb,latexsym}
\usepackage{overpic}

\newtheorem{theorem}{Theorem}[section]
\newtheorem{lemma}[theorem]{Lemma}
\newtheorem{corollary}[theorem]{Corollary}
\newtheorem{proposition}[theorem]{Proposition}

\theoremstyle{remark}
\newtheorem{definition}[theorem]{Definition}

\def\impart{\operatorname{Im}}

\newcommand{\Res}{\mathrm{Res}}

\numberwithin{equation}{section}

\begin{document}

\title{The global parametrix in the Riemann-Hilbert steepest descent
    analysis for orthogonal polynomials}
\author{Arno Kuijlaars\footnote{Department of Mathematics, Katholieke Universiteit Leuven,
Celestijnenlaan 200B, B-3001 Leuven, Belgium. email: arno.kuijlaars@wis.kuleuven.be.  
The work of the first author is supported by FWO-Flanders project
G.0427.09, by K.U.~Leuven research grant OT/08/33, by the Belgian
Interuniversity Attraction Pole P06/02, by the European Science
Foundation Program MISGAM, and by grant MTM2008-06689-C02-01 of the
Spanish Ministry of Science and Innovation.} \quad and \quad Man Yue
Mo\footnote{Department of Mathematics, University of Bristol,
Bristol BS8 1TW, U.K. email: m.mo@bristol.ac.uk. The second author
is supported by the EPSRC grant EP/G019843/1.}}

\date{\ }
\maketitle

\begin{abstract}
In the application of the Deift-Zhou steepest descent method to the
Riemann-Hilbert problem for orthogonal polynomials, a model
Riemann-Hilbert problem that appears in the multi-cut case is solved
with the use of hyperelliptic theta functions. We present here an
alternative approach which uses meromorphic differentials instead of
theta functions to construct the solution of the model
Riemann-Hilbert problem. By using this representation, we obtain a
new and elementary proof for the solvability of the model
Riemann-Hilbert problem.

\end{abstract}

\maketitle
\section{The global parametrix}

\subsection{Introduction}
The Deift-Zhou steepest descent method is a powerful technique in
the asymptotic analysis of Riemann-Hilbert problems that has been
successfully applied to numerous problems in integrable systems,
random matrix theory and orthogonal polynomials, see e.g.\
\cite{Deift}, \cite{DIZ}, \cite{DKMVZ}, \cite{DZ1}, \cite{KMM}. When
applying the steepest descent method as in \cite{DKMVZ}, one
performs a series of transformations to the $2\times 2$
matrix-valued Riemann-Hilbert problem for orthogonal polynomials to
eventually approximate it by a ``model Riemann-Hilbert problem''
which is also known as the ``global parametrix'' or the ``outer
parametrix''. The solution of the model Riemann-Hilbert (RH) problem
in \cite{Apt}, \cite{DIZ}, \cite{DKMVZ}, \cite{DIKZ}, \cite[Section
4.3]{KMM} and \cite{KK} uses Riemann theta functions on
hyperelliptic Riemann surfaces. In \cite{Kor04}, such approach was
generalized to higher dimensional Riemann-Hilbert problems with
quasi-permutation jump matrices and the solutions were expressed in
terms of Riemann theta functions together with the Szeg\"o kernel.
These constructions use various notions from Riemann surfaces, which,
although classical, require a fair amount of background in
algebraic geometry.

In this paper, we present an alternative approach to the solution of
the model RH problem. The approach in this paper is less
constructive since it does not lead to explicit formulas. However,
in many applications, such as the universality results of random
matrix theory, one is merely interested in the existence of a global
parametrix, rather than its explicit form. In such cases, the lack
of explicit formulas is not an issue. The present approach also
enables us to obtain a new elementary proof for the solvability of
the model Riemann-Hilbert problem.

On the other hand, our approach is conceptually rather simple (in
our opinion), and generalizes without too much effort from
hyperelliptic (i.e., two-sheeted) Riemann surfaces to multi-sheeted
Riemann surfaces which arise in the steepest descent analysis of
larger size RH problems associated with multiple orthogonal
polynomials \cite{VAGK}.
In \cite{DuKu}, the RH steepest descent analysis of a 4 x 4 matrix-valued
Riemann-Hilbert problem was done with the help of an associated Riemann
surface that is a four-sheeted cover of the Riemann sphere.
 The analysis in \cite{DuKu} was restricted to the
one-cut case (i.e., genus zero). The extension to the multi-cut case
was done in \cite{Mo} where both meromorphic differentials and
Riemann theta functions are used to solve the model RH problem.
Analyzing the approach in \cite{Mo} we found that it is also
possible to avoid the use of Riemann theta functions completely and
to use meromorphic differentials only. For the sake of clarity we
present this approach here for the case of the $2\times 2$
matrix-valued model RH problem as it arises in the steepest descent
analysis for orthogonal polynomials. We will use this construction
in the forthcoming work \cite{DKM} for a $4 \times 4$ matrix-valued
RH problem. See also \cite{BDK} for a similar situation
in a $3 \times 3$ context.

One of the main problems in solving the model Riemann-Hilbert problem
is the proof of its solvability. As pointed out in \cite{DIKZ} and
\cite{KK}, the model Riemann-Hilbert problem can be represented as
the monodromy problem of the Schlesinger equation.
To see whether the monodromy
problem is solvable, one can construct an isomonodromic tau function
\cite{JM}, \cite{Mal}, corresponding to the Schlesinger equation and the
monodromy problem will be solvable if the value of the tau function
is non-zero. However, the determination of the zeroes of the tau
function, known as the Malgrange divisor, is often a difficult task.
In \cite{KK} and \cite{Kor04}, it was shown that for the type of
model Riemann-Hilbert problem considered here, the isomonodromic tau
function is zero if and only if a certain theta function is zero.
Therefore to prove the solvability of the model Riemann-Hilbert
problem, one would need to study the theta divisor, which is a
highly transcendental object. By using the approach in this paper,
we were able to use much more elementary arguments to
show the existence of the global parametrix in the hyperelliptic
case, which allows us to avoid the theta divisor completely.

\subsection{The model Riemann-Hilbert problem}

The model RH problem that arises in the application of
the Deift-Zhou steepest descent method to orthogonal polynomials
is the following.

We are given $N$ intervals $[a_k,b_k]$, $k=1,\ldots, N$ on the
real line ordered so that $b_k < a_{k+1}$ for $k=1, \ldots, N-1$. We
also have $N-1$ real numbers $\alpha_k$ for $k=1, \ldots, N-1$, and
an integer $n$. The aim is then to construct a solution of the
following RH problem.

\begin{definition} \label{def:model}
The model Riemann-Hilbert problem is the following
RH problem for a $2 \times 2$-matrix
valued function $M: \mathbb C \setminus [a_1, b_N] \to \mathbb C^{2\times 2}$:
\begin{enumerate}
\item[\rm (a)] $M$ is analytic on $\mathbb C \setminus [a_1,b_N]$,
\item[\rm (b)] $M$ has jumps $M_+(x) = M_-(x) J_M(x)$ for $x \in [a_1, b_N]$
where
\begin{align} \label{eq:JM1}
    J_M(x) & = \begin{pmatrix} 0 & 1 \\ -1 & 0 \end{pmatrix}, \qquad
    \text{for } x \in (a_k,b_k),
\end{align}
for $k=1, \ldots, N$,
\begin{align} \label{eq:JM2}
    J_M(x) & = \begin{pmatrix} e^{-2\pi i n \alpha_k} & 0 \\
        0 & e^{2\pi i n \alpha_k} \end{pmatrix},
    \qquad \text{for } x \in (b_k, a_{k+1}),
    \end{align}
for $k=1, \ldots, N-1$,
\item[\rm (c)] $M(z) = I + O(1/z)$ as $z \to \infty$,
\item[\rm (d)] $M$ has at most fourth-root singularities near the
endpoints $a_k$ and $b_k$.
\end{enumerate}
\end{definition}

Clearly the jump condition \eqref{eq:JM2} in
the  model RH problem only depends on the value of
the numbers $n \alpha_k$ modulo the integers,
and so we may (and usually do) consider them to belong to $\mathbb{R} \slash \mathbb{Z}$.

The model RH problem was stated and solved in \cite{DKMVZ}.
The goal of this paper is to present an alternative
construction and to show that away from the endpoints $a_k$ and $b_k$, the
solution and its inverse are bounded in $n$.

As already mentioned, we use the hyperelliptic Riemann surface with
cuts along the intervals $[a_k,b_k]$. The main role is played by the
meromorphic differentials $\omega_P^{(\nu)}$, $\nu = 1,2$,
introduced in Definition \ref{def:1form}. The meromorphic
differential depends on $N-1$ points $P_1, \ldots, P_{N-1}$ on the
Riemann surface. The heart of the matter is Theorem
\ref{thm:Psibijection} which states that a suitably defined mapping
from $P_1, \ldots, P_{N-1}$ to a vector of $B$-periods is bijective.
This result allows us to take the points so that the $B$-periods are
exactly the numbers $2 \pi i n \alpha_k$, $k=1, \ldots, N-1$, that
appear in the jump condition \eqref{eq:JM2}. In Section
\ref{section3} we define the corresponding Abelian integrals, which
after exponentation lead to functions $v_j^{(\nu)}$, $j,\nu=1,2$,
that are used in the Definition \ref{def:M} of the solution of the
model RH problem.

Some further properties of $M$ are discussed in Section \ref{section4},
including the fact that $M$ and $M^{-1}$ are uniformly bounded
in $n$ if we stay away from the endpoints $a_k$ and $b_k$.
In the final Section \ref{section5} we present an alternative
construction for the second row of $M$, assuming that we know
the first row of $M$.

\section{Meromorphic differentials}
\label{section2}

The construction will be based on meromorphic differentials
(Abelian differentials of the third kind)
on the two-sheeted Riemann surface $\mathcal R$ for the equation
\begin{align} \label{eq:SurfaceEquation}
    w^2 = \prod_{k=1}^N (z-a_k)(z-b_k)
    \end{align}
which is obtained by gluing together two copies of
$\overline{\mathbb C} \setminus \bigcup_{k=1}^N [a_k,b_k]$ along the
cuts $[a_k,b_k]$ in the usual crosswise manner. The surface is
compact (we add a point at infinity to each sheet) and has genus
$N-1$.

We need a few standard facts about Riemann surfaces. Our main
reference is \cite{FK}.

\begin{figure}[t]
\centering
\begin{overpic}[width=15cm]{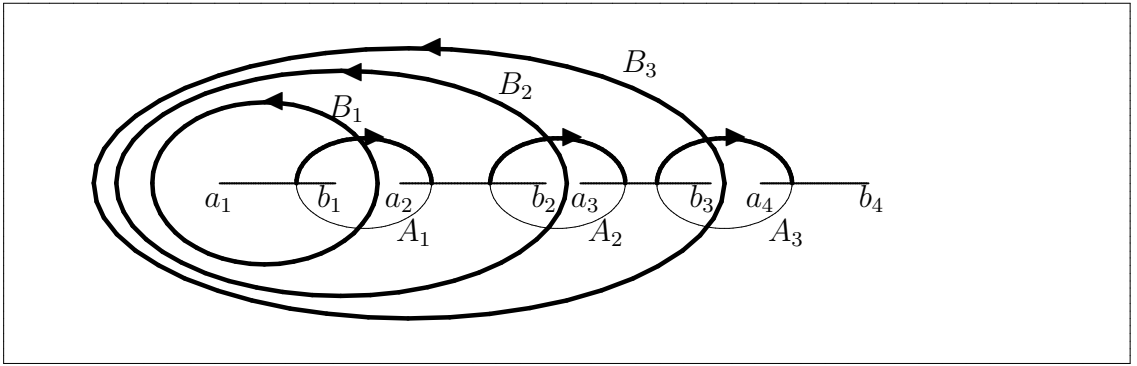}
    \put(18,14){$a_1$} \put(28,14){$b_1$}
    \put(34,14){$a_2$} \put(47,14){$b_2$}
    \put(50.5,14){$a_3$} \put(61,14){$b_3$}
    \put(66,14){$a_4$} \put(76,14){$b_4$}
    \put(35,11){$A_1$} \put(29,22){$B_1$}
    \put(52,11){$A_2$} \put(44,24){$B_2$}
    \put(68,11){$A_3$} \put(55,26){$B_3$}
    \put(23,22.7){$\blacktriangleleft$}
    \put(30,25.3){$\blacktriangleleft$}
    \put(37,27.5){$\blacktriangleleft$}
    \put(32,19.5){$\blacktriangleright$}
    \put(49.5,19.5){$\blacktriangleright$}
    \put(64,19.5){$\blacktriangleright$}
\end{overpic}
\caption{The canonical homology basis $(A_1, \ldots, A_{N-1},B_1,
\ldots, B_{N-1})$ for the case $N=4$. The thick contours are on the
first sheet and the thin contours are on the second sheet of the
Riemann surface $\mathcal R$. }
 \label{fig:canonicalhomology}
\end{figure}

We will now define the canonical homology basis on this Riemann
surface. We choose the cycles $A_j$ and $B_j$ for a canonical
homology basis
\[ (A_1, \ldots, A_{N-1}; B_1, \ldots, B_{N-1}) \]
as indicated in Figure \ref{fig:canonicalhomology} for the case
$N=4$. That is, $B_j$ is a cycle on the first sheet that encircles
the interval $[a_1, b_j]$ once in the counterclockwise direction. We
choose $B_j$ to be symmetric with respect to the real axis on the
first sheet. The cycle $A_j$ has a part in the upper half-plane of
the first sheet, a part in the lower half-plane of the second sheet,
and passes through the cuts $[a_j,b_j]$ and $[a_{j+1},b_{j+1}]$ with
an orientation also indicated in Figure \ref{fig:canonicalhomology}.

There is an anti-holomorphic involution $\phi$ on $\mathcal R$ which
maps $z$ to $\overline{z}$ on the same sheet. The set of fixed
points of $\phi$ are the real ovals, which are $N$ closed contours
$\Gamma_j$, $j=0, \ldots, N-1$, where for $j=1, \ldots, N-1$,
$\Gamma_j$ is the union of the two intervals $[b_j,a_{j+1}]$ from
both sheets and $\Gamma_0$ is unbounded and contains the intervals
from $a_1$ and $b_N$ to the two points at infinity. Thus
\begin{align} \label{def:gammaj}
    \Gamma_j & = \{ (z,w) \in \mathcal R \mid b_j \leq z \leq a_{j+1} \},
        \quad j=1, \ldots, N-1.
    \end{align}
The cycle $A_j$ is homotopic to $\Gamma_j$ but
we choose $A_j$ to be disjoint from $\Gamma_j$.

For each $j = 1, \ldots, N-1$ we choose a point $P_j \in \Gamma_j$.
We are going to associate with the $N-1$ points
$(P_1,\ldots,P_{N-1})$ and an index $\nu \in \{1,2\}$ a meromorphic
differential $\omega = \omega_P^{(\nu)}$.

\begin{definition} \label{def:1form}
The meromorphic differential $\omega =\omega_P^{(\nu)}$ associated
with $(P_1,\ldots,P_{N-1}) \in \Gamma_1 \times \cdots \times
\Gamma_{N-1}$ and $\nu \in \{1,2\}$, is defined by the following
properties:
\begin{itemize}
\item $\omega_P^{(\nu)}$ has simple poles at the points
   $a_1, b_1, \ldots, a_N, b_N$, $P_1, \ldots, P_{N-1}$ with residues
\begin{align} \label{eq:res1}
    \Res(\omega_P^{(\nu)}, a_j)  & = \Res(\omega_P^{(\nu)}, b_j) = -\tfrac{1}{2},
        \qquad j=1,\ldots, N, \\
        \label{eq:res2}
    \Res(\omega_P^{(\nu)}, P_j) & = 1,   \qquad j=1, \ldots, N-1,
    \end{align}
\item $\omega_P^{(1)}$ has a simple pole at $\infty_2$ (the point
at infinity on the second sheet) and $\omega_P^{(2)}$ has a simple
pole at $\infty_1$ (the point at infinity on the first sheet) with residue
\begin{align} \label{eq:res3}
    \Res(\omega_P^{(1)}, \infty_2) & = \Res(\omega_P^{(2)}, \infty_1) =1.
    \end{align}
\item The differential $\omega_P^{(\nu)}$ is holomorphic elsewhere.
\item The $A$-periods satisfy:
\begin{align} \label{eq:Aperiods}
    \int_{A_j} \omega_P^{(\nu)} = 0, \qquad j=1, \ldots, N-1.
    \end{align}
Note that the points $P_j$ are not on any of the $A$-cylces.
\end{itemize}
\end{definition}

The meromorphic differential $\omega_P^{(\nu)}$ exists and is
uniquely defined by the properties
\eqref{eq:res1}--\eqref{eq:Aperiods}. Indeed, a simple count shows
that the sum of the residues in \eqref{eq:res1}--\eqref{eq:res3} is
equal to zero which is a necessary and sufficient condition for the
meromorphic differential to exist, see for example Theorem II.5.3 in
\cite{FK}. The residue conditions determine the meromorphic
differential up to a holomorphic differential. The vector space of
holomorphic differentials has dimension $N-1$, and the $N-1$
conditions in \eqref{eq:Aperiods} determine the meromorphic
differential uniquely.

If one or more of the $P_j$'s coincide with a branch point, then the
residue conditions \eqref{eq:res1}--\eqref{eq:res2} have to be modified appropriately.
For example, if $P_j = b_j$ then
\[ \Res(\omega_P^{(\nu)}, P_j) = \tfrac{1}{2}. \]
In this way, the meromorphic differential $\omega_P^{(\nu)}$ depends
continuously on the $P_j$'s. This fact will play a role in the proof of Proposition \ref{prop:Psicontinuous} below.

The anti-holomorphic involution $\phi$ can be used to map
$\omega_P^{(\nu)}$ to a meromorphic differential
$\phi^{\#}(\omega_P^{(\nu)})$ in an obvious way. If $\omega_P^{(\nu)}$
is equal to $f_j(z) dz$ for some meromorphic function $f_j$ on sheet
$j$ for $j=1,2$, then $\phi^{\#}(\omega_P^{(\nu)})$ is equal to
\[ \overline{f_j(\overline{z})} \, dz \]
on sheet $j$. A crucial property is that $\omega^{(\nu)}_P$ is
invariant under the map $\phi^{\#}$.

\begin{lemma} \label{lem:omegasymmetry}
For every $(P_1, \ldots, P_{N-1}) \in \Gamma_1 \times \cdots \times
\Gamma_{N-1}$ and $\nu \in \{1,2\}$, we have
\begin{equation} \label{eq:phiomegaP}
    \omega_P^{(\nu)} = \phi^{\#}(\omega_P^{(\nu)}).
    \end{equation}
\end{lemma}
\begin{proof}
Since all poles of $\omega_P^{(\nu)}$ are invariant under $\phi$, the
meromorphic differential $\phi^{\#}(\omega_P^{(\nu)})$ has the same
(simple) poles as $\omega_P^{(\nu)}$ with the same residues. For the
$A$-periods we have by definition of $\phi$,
\[ \int_{A_j} \phi^{\#}(\omega_P^{(\nu)}) = \overline{\int_{\phi(A_j)} \omega_P^{(\nu)}},
    \qquad j =1, \ldots, N-1. \]
The cycle $\phi(A_j)$ is homotopic to $A_j$ in $\mathcal R$. In the
process of deforming $\phi(A_j)$ to $A_j$ we pick up residue
contributions from the poles of $\omega_P^{(\nu)}$ at $b_j$, $P_j$ and
$a_{j+1}$. Since the combined residue is $-\tfrac{1}{2} + 1 -
\tfrac{1}{2} = 0$, it follows that
\[ \int_{A_j} \phi^{\#}(\omega_P^{(\nu)}) = \overline{\int_{A_j} \omega_P^{(\nu)}} = 0. \]
Therefore $\phi^{\#}(\omega_P^{(\nu)})$ has all the properties that
characterize $\omega_P^{(\nu)}$ and the lemma follows.
\end{proof}

\subsection{The map $\Psi^{(\nu)}$ from $(P_1, \ldots, P_{N-1})$ to the $B$-periods}
The differential $\omega_P^{(\nu)}$ has a vector of $B$-periods and
it is convenient for us to divide by $2 \pi i$. So we define
$(\beta_1, \ldots, \beta_{N-1})$ with
\begin{equation} \label{eq:betak}
    \beta_k = \frac{1}{2\pi i} \int_{B_k} \omega_P^{(\nu)}.
    \end{equation}
Then by mapping the points $(P_1,\ldots,P_{N-1})$ to the $B$-periods
of the differential $\omega_P^{(\nu)}$, we obtain a map from
$\Gamma_1\times\cdots\times\Gamma_{N-1}$ into $\mathbb{C}^{N-1}$. We
will show that this map is well-defined  from
$\Gamma_1\times\cdots\times\Gamma_{N-1}$ to $\left(\mathbb{R} \slash
\mathbb{Z}\right)^{N-1}$.
\begin{proposition} \label{prop:Psicontinuous}
For $\nu =1,2$, the map
\begin{multline} \label{def:Psi}
    \Psi^{(\nu)} : \Gamma_1 \times \cdots \times \Gamma_{N-1} \to \left( \mathbb R \slash \mathbb Z \right)^{N-1} : \\
    (P_1, \ldots, P_{N-1})  \mapsto
    (\beta_1, \ldots, \beta_{N-1}) =
       \frac{1}{2\pi i} \left(\int_{B_1} \omega_P^{(\nu)}, \ldots, \int_{B_{N-1}} \omega_P^{(\nu)} \right)
       \end{multline}
 is well-defined and continuous.
\end{proposition}
\begin{proof}
Since the meromorphic differential $\omega_P^{(\nu)}$ depends continuously
on the $P_j$'s, we then also have that the $\beta_k$ from \eqref{eq:betak}
is well-defined and continuous  in the $P_j$'s,  unless $P_k$ is on $B_k$.

Recall that $B_k$ intersects the interval $(a_k,
b_{k+1})$ in one point on the first sheet. The value of $\beta_k$ then makes a jump
when  $P_k \in \Gamma_k$ passes through this intersection
point. As the residue of $\omega_P^{(\nu)}$ at the pole $P_k$ is an integer
(in fact, it is $1$, see \eqref{eq:res2}), the jump in $\beta_k$ is by an integer
value. Since we consider values modulo $\mathbb Z$, the map $\Psi^{(\nu)}$
is thus well-defined and continuous
from $\Gamma_1\times\cdots\times\Gamma_{N-1}$ into $\left( \mathbb C
\slash \mathbb Z \right)^{N-1}$.

Now let us show that the $\beta_k$ are real and hence $\Psi^{(\nu)}$
is really a map into $\left( \mathbb R \slash \mathbb Z
\right)^{N-1}$. Recall that the cycle $B_k$ is chosen to be
symmetric with respect to the real axis. Therefore $\phi(B_k) = -
B_k$. Since $\omega_P^{(\nu)} = \phi^{\#}(\omega_P^{(\nu)})$ by
Lemma \ref{lem:omegasymmetry}, we then have
\begin{align*}
    \int_{B_k} \omega_P^{(\nu)} & = \int_{B_k} \phi^{\#}(\omega_P^{(\nu)})
    = \overline{\int_{\phi(B_k)} \omega_P^{(\nu)}}
     = \overline{\int_{-B_k} \omega_P^{(\nu)}}
    = - \overline{\int_{B_k} \omega_P^{(\nu)}}.
    \end{align*}
Thus $\int_{B_k} \omega_P^{(\nu)}$ is purely imaginary and so
$\beta_k$ is real indeed.
\end{proof}

The main result of this section is that the map \eqref{def:Psi} is a
bijection. This would imply that there exists a unique set of points
$(P_1^{(\nu)},\ldots,P_{N-1}^{(\nu)})$ such that the $B$-periods of
the differential $\omega_P^{(\nu)}$ are given by $2\pi i n
\alpha_j$. The bijectivity proof relies on the fact that the divisor
corresponding to any choice of points $P_1, \ldots, P_{N-1}$ with
$P_j \in \Gamma_j$ is non-special.

We use additive notation for divisors and we write
\begin{equation} \label{eq:divisorD}
    D = \sum_{j=1}^{N-1} P_j.
    \end{equation}
A divisor \eqref{eq:divisorD}  is special if there exists a
non-constant holomorphic function on $\mathcal R \setminus \{ P_1,
\ldots, P_{N-1} \}$ with at most simple poles at the points $P_j$.
By the Riemann-Roch theorem, the divisor \eqref{eq:divisorD} is
special if and only if there exists a non-zero holomorphic
differential with zeros at each of the points $P_j$ for $j=1,
\ldots, N-1$.

\begin{lemma} \label{lem:nonspecial}
(see also Statement 1 in \cite{Kor}) If $P_j \in \Gamma_j$ for each
$j=1, \ldots, N-1$, then the divisor \eqref{eq:divisorD} is
non-special.
\end{lemma}
\begin{proof}
The holomorphic differentials on the hyperelliptic Riemann surface
defined by \eqref{eq:SurfaceEquation} are of the form
\[ \frac{p(z)}{w} dz, \]
where $p$ is a polynomial of degree $\leq N-2$. Therefore the zeros
of a non-zero holomorphic differential project onto at most $N-2$
points in the complex $z$-plane. The points $P_j \in \Gamma_j$,
$j=1, \ldots, N-1$, project onto $N-1$ distinct points, and so there
can be no non-zero holomorphic differential with a zero at each of
the $P_j$'s.
\end{proof}

In the proof of Theorem~\ref{thm:Psibijection}, which is the main result of
this section, we also need the invariance of domain theorem of Brouwer \cite{Bro},
which is a classical result from  topology.
See e.g.\ \cite[section XVII 3]{Dug} or \cite{Ful} for more recent accounts.
We state the theorem here for the reader's convenience.

\begin{theorem} \label{thm:invariance} \textbf{(invariance of domain)}
    If $U$ is an open subset of $\mathbb R^n$ and
    $f : U \to \mathbb R^n$ is an injective continuous map,
    then $f$ is open (i.e., $f$ maps open subsets of $U$
    to open subsets of $\mathbb R^n$).
\end{theorem}
Of course, the theorem readily extends to injective continuous maps between manifolds
of the same dimension, which is what we will use in the proof
of Theorem \ref{thm:Psibijection}.

\begin{theorem} \label{thm:Psibijection}
The map $\Psi^{(\nu)}$ defined in \eqref{def:Psi} is a bijection
from $\Gamma_1 \times \cdots \times \Gamma_{N-1}$ to $\left( \mathbb
R \slash \mathbb Z \right)^{N-1}$.
\end{theorem}

\begin{proof}
We first prove that $\Psi^{(\nu)}$ is injective. Suppose $(P_1,
\ldots, P_{N-1})$ and $(Q_1, \ldots, Q_{N-1})$ are in $\Gamma_1
\times \cdots \times \Gamma_{N-1}$ so that
\[ \Psi^{(\nu)}(P_1, \ldots, P_{N-1}) = \Psi^{(\nu)}(Q_1, \ldots, Q_{N-1}). \]
 Let
$\omega_P^{(\nu)}$ and $\omega_Q^{(\nu)}$ be the corresponding
meromorphic differentials. Then $\omega_Q^{(\nu)} -
\omega_P^{(\nu)}$ has poles with residues $\pm 1$ at the points
$Q_j$ and $P_j$ only and all periods are integer multiples of $2\pi
i$. Let $P_0$ be a given base point different from any of the
$P_j$'s and $Q_j$'s. It then follows that
\[ \exp\left(\int_{P_0}^z \left(\omega_Q^{(\nu)} - \omega_P^{(\nu)}\right)\right),
    \qquad z \in \mathcal R, \]
is a meromorphic function on $\mathcal R$ with only possible poles
at $P_1, \ldots, P_{N-1}$. Since the divisor $D = \sum_{j=1}^{N-1}
P_j$ is non-special, see Lemma \ref{lem:nonspecial}, the
meromorphic function is a constant, which implies that
$\omega_Q^{(\nu)} = \omega_P^{(\nu)}$. Hence the $Q_j$'s and the $P_j$'s coincide: $(Q_1, \ldots, Q_{N-1})
= (P_1, \ldots, P_{N-1})$ and therefore $\Psi^{(\nu)}$ is injective.

To prove surjectivity we now note that $\Psi^{(\nu)}$ is an
injective and continuous (by Proposition \ref{prop:Psicontinuous})
map from the $N-1$-dimensional manifold
$\Gamma_1 \times \cdots \times \Gamma_{N-1}$ to the
$N-1$-dimensional manifold $(\mathbb R \slash \mathbb Z)^{N-1}$.
Thus $\Psi^{(\nu)}$ is open by Theorem~\ref{thm:invariance}.
It follows that
$\Psi^{(\nu)}(\Gamma_1 \times \cdots \times \Gamma_{N-1})$ is a
subset of $(\mathbb R \slash \mathbb Z)^{N-1}$ that is both open
(since $\Psi^{(\nu)}$ is open) and compact (since $\Psi^{(\nu)}$ is
continuous and $\Gamma_1 \times \cdots \times \Gamma_{N-1}$ is
compact). Since $(\mathbb R \slash \mathbb Z)^{N-1}$ is connected it
follows that $\Psi^{(\nu)}$ is surjective.
\end{proof}

\section{Construction of $M$}
\label{section3}

It will be a consequence of Theorem \ref{thm:Psibijection} that we
can construct the matrix $M$ that solves the model Riemann-Hilbert problem.

By Theorem \ref{thm:Psibijection} there exist $P_j^{(\nu)} \in
\Gamma_j$ for $j=1, \ldots, N-1$, $\nu =1,2$, so that
\begin{equation} \label{eq:Psiimage}
    \Psi^{(\nu)}(P_1^{(\nu)}, \ldots, P_{N-1}^{(\nu)}) = \left(n \alpha_1, \ldots, n \alpha_{N-1}\right)
    \end{equation}
where each $n \alpha_j$ is considered modulo $\mathbb Z$. Let
$\omega_P^{(\nu)}$ be the corresponding meromorphic differential. We then
have that
\begin{equation} \label{eq:Bperiods}
     \frac{1}{2\pi i} \int_{B_k} \omega_P^{(\nu)} \equiv n \alpha_k \quad \mod \mathbb Z,
     \qquad \text{for } k = 1, \ldots, N-1.
     \end{equation}

\subsection{Abelian integrals}
For $\nu =1,2$, we choose the base point
\begin{equation} \label{eq:basepoint}
    P_0 = \infty_{\nu}
\end{equation}
and we define the functions $u_1^{(\nu)}(z)$, $u_2^{(\nu)}(z)$
of a complex variable $z$ as follows.

\begin{definition}\label{def:abel}
Let $z \in \mathbb C \setminus \mathbb R$.
\begin{enumerate}
\item[\rm (a)] To define $u_{j}^{(\nu)}$ with $j=\nu$, we consider $z$ as a point on the
$j$th sheet of the Riemann surface. We define
\begin{equation} \label{def:u1z}
    u_{j}^{(\nu)}(z) = \int_{P_0}^z \omega_P^{(\nu)},  \qquad z \in \mathbb C \setminus \mathbb R,
    \quad j = \nu,
    \end{equation}
where the path of integration is on the $j$th sheet of the
Riemann surface and it does not intersect the real line, except for
the initial point $P_0$.
\item[\rm (b)] To define $u_{j}^{(\nu)}(z)$ with $j \neq \nu$, we consider $z$ as a point on the
$j$th sheet. We define
\begin{equation} \label{def:u2z}
    u_j^{(\nu)}(z) = \int_{P_0}^z \omega_P^{(\nu)},  \qquad z \in \mathbb C \setminus \mathbb
    R,\quad j \neq \nu,
    \end{equation}
where now the path of integration is as follows. If $\impart z > 0$
($ \impart z < 0$) then the path starts in the lower (upper)
half-plane of the $\nu$th sheet and passes to the $j$th sheet
via a cut $(a_k,b_k)$. It then stays in the upper (lower) half-plane
of the $j$th sheet.
\end{enumerate}
\end{definition}

Since $\omega_P^{(\nu)}$ has vanishing $A$-periods, as well as
vanishing $\phi^{\#}(A)$ periods, it does not matter which cut
$(a_k,b_k)$ is taken, and so $u_j^{(\nu)}(z)$  in \eqref{def:u2z} is
uniquely defined.

The functions $u_j^{(\nu)}$ are analytic on $\mathbb C \setminus \mathbb
R$ with the following jumps on $\mathbb R$.

\begin{lemma} \label{lem:u1u2jumps}
The functions $u_j^{(\nu)}$, $j,\nu =1,2$ satisfy the following jump
conditions for $x \in \mathbb R$.
\begin{enumerate}
\item[\rm (a)] For $x \in (a_k, b_k)$ with $k=1, \ldots, N$, we have
\begin{align}  \label{jumpu1a}
    u_{1,+}^{(\nu)}(x) & = u_{2,-}^{(\nu)}(x), \\ \label{jumpu1b}
    u_{2,+}^{(\nu)}(x) & = u_{1,-}^{(\nu)}(x).
\end{align}
\item[\rm (b)] For $x < a_1$ or $x > b_N$ we have
\begin{align} \label{jumpu2a}
    u_{j,+}^{(\nu)}(x) & = u_{j,-}^{(\nu)}(x), \qquad j = \nu, \\ \label{jumpu2b}
    u_{j,+}^{(\nu)}(x) & \equiv u_{j,-}^{(\nu)}(x) + \pi i \quad \mod 2 \pi i \, \mathbb Z,\quad j\neq \nu.
\end{align}
\item[\rm (c)] For $z \in \mathbb C \setminus \bigcup_{k=1}^N [a_k, b_k]$,  we
use $P_j(z)$ to denote the point on the $j$th sheet of $\mathcal R$
that corresponds to $z$. Then we have, for $x \in (b_k, a_{k+1})$
with $k=1,\ldots, N-1$,
\begin{align} \label{jumpu3a}
    u_{1,+}^{(1)}(x) & \equiv u_{1,-}^{(1)}(x) - 2\pi i n \alpha_k,
    \qquad \text{if } P_1(x) \neq P_k^{(1)}, \\
    \label{jumpu3b}
    u_{2,+}^{(1)}(x) & \equiv u_{2,-}^{(1)}(x) + 2 \pi i n \alpha_k + \pi i
    \qquad \text{if } P_2(x) \neq P_k^{(1)}, \\
    \label{jumpu3c}
    u_{1,+}^{(2)}(x) & \equiv u_{1,-}^{(2)}(x) - 2\pi i n \alpha_k + \pi i
    \qquad \text{if } P_1(x) \neq P_k^{(2)}, \\
    \label{jumpu3d}
    u_{2,+}^{(2)}(x) & \equiv u_{2,-}^{(2)}(x) + 2 \pi i n \alpha_k,
    \qquad \text{if } P_2(x) \neq P_k^{(2)},
\end{align}
The equalities \eqref{jumpu3a}--\eqref{jumpu3d} are valid modulo $2\pi i \, \mathbb Z$.
\end{enumerate}
\end{lemma}

\begin{proof}
The properties  \eqref{jumpu1a} and \eqref{jumpu1b} follow
immediately from the definition of $u_j^{(\nu)}$.

Let $\Delta^{(\nu)}$ be the set of poles for $\omega_P^{(\nu)}$, that is,
\begin{equation*}
\Delta^{(\nu)}=\{a_1,b_1,\ldots,a_N,b_N,P_1^{(\nu)},\ldots,P_{N-1}^{(\nu)},
\infty_j\},\quad j\neq \nu.
\end{equation*}

For $x < a_1$ and $x > b_N$, we have
\[ u_{j,+}^{(\nu)}(x) - u_{j,-}^{(\nu)}(x) = \oint_{C} \omega_P^{(\nu)} \]
where $C$ is a closed contour on $\mathcal R \setminus\Delta^{(\nu)}$.
For $j=\nu$, the contour $C$ is contractible in
$\mathcal{R} \setminus \Delta^{(\nu)}$ and \eqref{jumpu2a} follows.
When $j\neq \nu$, we choose the contour to pass through the cut $[a_1,b_1]$
in case $x < a_1$, and through the cut $[a_N,b_N]$ in case $x > b_N$.
Then $C$ is contractible to a small loop around $a_1$ (in case $x < a_1$)
or around $b_N$ (in case $x > b_N$). Since the residues of $\omega_P^{(\nu)}$ at
$a_1$ and $b_N$ are $-\frac{1}{2}$ we have in either case
\[ \oint_{C} \omega_P^{(\nu)} = \pi i \qquad \mod 2 \pi i, \]
and  \eqref{jumpu2b} follows.

Let  $x \in (b_k, a_{k+1})$, $x \neq z(P_k^{(\nu)})$, for some $k=1,
\ldots, N-1$. Then we have
\[ u_{j,+}^{(\nu)}(x) - u_{j,-}^{(\nu)}(x) = \oint_{C} \omega_P^{(\nu)} \]
where $C$ is again a closed contour on $\mathcal R \setminus \Delta^{(\nu)}$.
When $j = \nu = 1$, then $C$ is on the first sheet and can be
deformed into $-B_k$ and \eqref{jumpu3a} follows because of \eqref{eq:Bperiods}.

When $j= \nu = 2$, then $C$ is a closed contour on the second sheet,
and it  is homotopic to $B_k$ in $\mathcal R$, but the
deformation will pick up residue contributions from the poles at
$a_1,\ldots,a_k, b_1,\ldots,b_k$,
$P_1^{(\nu)},\ldots,P_{k-1}^{(\nu)}$ and possibly $P_k^{(\nu)}$
depending on its position. Since the combined residues of
$\omega_P^{(\nu)}$ at these poles is an integer number, the poles do
not contribute (modulo $2\pi i$) and we obtain \eqref{jumpu3d} from
\eqref{eq:Bperiods}.

When $j\neq \nu$, the closed loop $C$ is on both sheets, For $j=2$,
$\nu=1$, we choose $C$ so that it passes through the cut
$[a_k,b_k]$. Then $C$ is homotopic to $B_k$ in $\mathcal R$, but a
deformation from $C$ to $B_k$ will pick up a residue contribution
from $b_k$ and possibly from $P_k^{(\nu)}$. The combined residue is $-1/2$
or $+1/2$, and this leads to \eqref{jumpu3b}.

Finally, for $j=1$, $\nu=2$, we choose $C$ so that it passes through
the cut $[a_1,b_1]$. Then $C$ is homotopic to  $-B_k$ in $\mathcal
R$, and the deformation from $C$ to $-B_k$ picks up a residue
contribution at $a_1$ and possibly $P_k^{(\nu)}$. The combined residue is
$-1/2$ or $+1/2$, and we obtain \eqref{jumpu3c}.
\end{proof}

Note that in part (c) of Lemma \ref{lem:u1u2jumps} we excluded
the case $P_j(x) = P_k^{(\nu)}$ since $P_k^{(\nu)}$ is a pole of $\omega_P^{(\nu)}$
and so the limiting values $u_{j,+}^{(\nu)}(x)$ and $u_{j,-}^{(\nu)}(x)$ do not
exist if $P_k^{(\nu)}$ is on the $j$th sheet, see also part (b) of Lemma \ref{lem:u1u2behavior}.

The behavior near all poles of $\omega_P^{(\nu)}$ is stated in the following lemma.
We use $z(P)$ to denote the $z$-coordinate of a point $P=(z,w)$ on the Riemann
surface \eqref{eq:SurfaceEquation}.

\begin{lemma} \label{lem:u1u2behavior}
We have
\begin{enumerate}
\item[\rm (a)] for $j, \nu = 1,2$ and $k=1, \ldots, N$,
\begin{align*}
    u_j^{(\nu)}(z) & = -\tfrac{1}{4} \log(z-a_k) + O(1), \qquad \text{as } z \to a_k, \\
    u_j^{(\nu)}(z) & = -\tfrac{1}{4} \log(z-b_k) + O(1), \qquad \text{as } z \to b_k,
\end{align*}
\item[\rm (b)] if $P_k^{(\nu)}$ is on the $j$th sheet of the Riemann surface, then
\begin{align} \label{eq:ujatPknu}
    u_j^{(\nu)}(z) = \log(z- z(P_k^{(\nu)}))  + O(1) \qquad \text{as } z \to z(P_k^{(\nu)}),
    \end{align} \item[\rm (c)] as $z \to \infty$
\begin{align*}
    u_j^{(\nu)}(z) & = O(1/z),  \qquad \text{if } j = \nu \\
    u_j^{(\nu)}(z) & = - \log z + O(1), \quad \text{if } j \neq \nu.
    \end{align*}
\end{enumerate}
\end{lemma}
\begin{proof}
The fact that $u_j^{(\nu)}(z) = O(1/z)$ as $z \to \infty$ in case $j = \nu$ follows directly
from \eqref{def:u1z} since $P_0 = \infty_{\nu}$. The other
statements of the lemma follow from \eqref{def:u1z}--\eqref{def:u2z}
and the residue conditions in \eqref{eq:res1}--\eqref{eq:res3}.
\end{proof}

In Lemma \ref{lem:u1u2behavior} we implicitly assumed that the point
$P_k^{(\nu)}$ is different from $b_k$ and $a_{k+1}$.
If for example, $P_k^{(\nu)} = b_k$, then the residue
of $\omega_P^{(\nu)}$ at $b_k$ is equal to $+\tfrac{1}{2}$, and in part (a)
of Lemma \ref{lem:u1u2behavior} we get
\[ u_j^{(\nu)}(z) = \tfrac{1}{4} \log(z-b_k) + O(1),
    \qquad \text{as } z \to b_k. \]
The modifications that are needed when one or more of the
$P_k^{(\nu)}$ coincide with an endpoint are obvious,
and we will not specify them in the rest of the paper.

\subsection{Exponential of the Abelian integrals}

Now we exponentiate the functions $u_j^{(\nu)}$.
\begin{definition} \label{def:v1v2}
We define $v_j^{(\nu)}$, $j,\nu=1,2$ by
\begin{equation} \label{eq:v1v2}
    v_j^{(\nu)}(z) = \exp \left( u_j^{(\nu)}(z) \right), \qquad z \in \mathbb C \setminus \mathbb R.
    \end{equation}
    \end{definition}

Then the functions $v_j^{(\nu)}(z)$ are analytic in $\mathbb C \setminus
\mathbb R$ with the following jumps on $\mathbb R$.
\begin{corollary} \label{cor:v1v2jumps}
The vectors $(v_1^{(\nu)},v_2^{(\nu)})$ satisfy the following jump
conditions on $\mathbb{R} \setminus \{a_1,b_1, \ldots, a_N, b_N\}$,
\begin{equation}\label{eq:vjump}
    (v_1^{(\nu)},v_2^{(\nu)})_+ = (v_1^{(\nu)},v_2^{(\nu)})_- J_v^{(\nu)}, \qquad \nu =1,2,
\end{equation}
where
\begin{align} \label{eq:Jvnu1}
    J_v^{(\nu)}(x) & = \begin{pmatrix} 0 & 1 \\ 1 & 0 \end{pmatrix} \qquad \text{for } a_k <  x < b_k,
        \quad k = 1, \ldots, N, \\
        \label{eq:Jvnu2}
    J_v^{(\nu)}(x) & = (-1)^{\nu-1} \begin{pmatrix} 1 & 0 \\ 0 & -1 \end{pmatrix} \qquad
    \text{for } x < a_1 \text{ or } x > b_N, \\
        \label{eq:Jvnu3}
    J_v^{(\nu)}(x) & =  (-1)^{\nu-1} \begin{pmatrix} e^{-2\pi i n \alpha_k} & 0 \\
        0 & - e^{2\pi in \alpha_k} \end{pmatrix} \qquad
    \text{for } b_k <  x < a_{k+1},
\end{align}
for $k=1, \ldots, N-1$.
\end{corollary}

\begin{proof} The jumps follow directly from Lemma \ref{lem:u1u2jumps}
and Definition \ref{def:v1v2} in case $x \neq z(P_k^{(\nu)})$.

If $x = z(P_k^{(\nu)})$ and $P_k^{(\nu)}$ is on the $j$th sheet,
then Lemma \ref{lem:u1u2jumps} does not apply to $u_j$.
However, in that case we find by part (b) of Lemma \ref{lem:u1u2behavior}
and Definition \ref{def:v1v2} that
\[ v_{j,+}^{(\nu)}(x) = v_{j,-}^{(\nu)}(x) = 0, \]
and \eqref{eq:vjump}, \eqref{eq:Jvnu3} is also valid.

\end{proof}

From Lemma \ref{lem:u1u2behavior} and
\eqref{eq:v1v2} we find the following behavior near the poles of
$\omega_P^{(\nu)}$.
\begin{corollary} \label{cor:v1v2behavior}
We have
\begin{enumerate}
\item[\rm (a)] for $j, \nu = 1,2$ and $k=1, \ldots, N$,
\begin{align*}
    v_j^{(\nu)}(z) & = O\left((z-a_k)^{-1/4}\right), \qquad  \text{as } z \to a_k, \\
    v_j^{(\nu)}(z) & = O\left((z-b_k)^{-1/4}\right), \qquad \text{as } z \to b_k,
\end{align*}
\item[\rm (b)]  if $P_k^{(\nu)}$ is on the $j$th sheet of $\mathcal R$, then
\begin{equation} \label{eq:vjatPk}
    v_j^{(\nu)}(z) = O (z- z(P_k^{(\nu)}))  \qquad \text{as } z \to z(P_k^{(\nu)}),
    \end{equation}
\item[\rm (c)] as $z \to \infty$
\begin{align*}
    v_j^{(\nu)}(z) & = 1 + O(1/z), \quad \text{if } j = \nu \\
    v_j^{(\nu)}(z) & = O(1/z), \qquad \text{if } j \neq \nu.
    \end{align*}
\end{enumerate}
\end{corollary}

Note that by \eqref{eq:vjatPk} the function $v_j^{(\nu)}$,
has a zero at $z(P_k^{\nu})$ in case $P_k^{(\nu)}$ is on the $j$th sheet.

\subsection{The parametrix $M$}
To construct the solution $M$ of the model Riemann-Hilbert problem
in Definition \ref{def:model}, we only need a trivial modification
of the functions $v_j^{(\nu)}$.
\begin{definition} \label{def:M}
We define
\begin{equation} \label{eq:defM}
    M(z) = \left\{ \begin{aligned}
    \begin{pmatrix} v_1^{(1)}(z) & v_2^{(1)}(z)   \\
      -v_1^{(2)}(z)  &  v_2^{(2)}(z) \end{pmatrix}, \qquad \text{if } \impart z > 0, \\
    \begin{pmatrix} v_1^{(1)}(z) & -v_2^{(1)}(z)   \\
      v_1^{(2)}(z) &  v_2^{(2)}(z) \end{pmatrix}, \qquad \text{if } \impart z < 0.
        \end{aligned} \right.
        \end{equation}
\end{definition}

From \eqref{eq:vjump}--\eqref{eq:Jvnu3}, it is easy to verify that
$M(z)$ does indeed satisfy the jump conditions
\eqref{eq:JM1}--\eqref{eq:JM2} of the RH problem in Definition
\ref{def:model}. The asymptotic condition $M(z) = I + O(1/z)$ as
$z \to \infty$ holds because of part (c) of Corollary
\ref{cor:v1v2behavior}. Part (a) of Corollary
\ref{cor:v1v2behavior} shows that $M$ has at most fourth root
singularities at the endpoints $a_k$ and $b_k$. So we have proved
the main result of this paper:
\begin{theorem} \label{thm:Msolution}
The matrix-valued function $M(z)$ defined by \eqref{eq:defM}
satisfies the model Riemann-Hilbert problem of Definition
\ref{def:model}.
\end{theorem}

\section{Properties of $M$} \label{section4}
We collect here some further properties of $M$ that are useful
in applications.

\subsection{Uniqueness of the solution}
The first two properties are standard, see e.g. \cite{Deift}
\begin{proposition} \label{prop:detM}
For every $z \in \mathbb C \setminus [a_1,b_N]$ we have
\[ \det M(z) = 1. \]
\end{proposition}
\begin{proof}
Since the jump matrices in the model Riemann-Hilbert have
determinant $1$, the function $z \mapsto \det M(z)$ has no jump
discontinuities in $\mathbb{C}$, and so it has an analytic
extension to $\mathbb C \setminus \{a_1, b_1, \ldots, a_N, b_N\}$.
From part (d) in Definition \ref{def:model}, we conclude that
$\det M(z)$ can have at most square-root singularities at the
endpoint $a_j, b_j$, and therefore these isolated singularties are
removable. Thus $z \mapsto \det M(z)$ is an entire function, which
by the asymptotic condition satisfies $\det M(z) = 1 + O(z^{-1})$
as $z \to \infty$. Then the proposition follows, by Liouville's
theorem.
\end{proof}

\begin{proposition} \label{prop:Munique}
The solution $M$ of the model RH problem is unique.
\end{proposition}
\begin{proof}
Let $\widetilde{M}$ be a second solution of the model RH problem.
By Proposition \ref{prop:detM}, we have that $M(z)$ is invertible
for every $z \in \mathbb C \setminus [a_1,b_N]$. Then
\[ H(z) = \widetilde{M}(z) M(z)^{-1}, \qquad z \in \mathbb C \setminus \mathbb [a_1,b_N] \]
is well-defined and analytic. Since $\widetilde{M}$ and $M$ satisfy the same jump
conditions, it follows that $H_+(x) = H_-(x)$ for every $x \in [a_1,b_N] \setminus
\{ a_1, b_1, \ldots, a_N,b_N \}$. The entries of $\widetilde{M}$ and $M^{-1}$
have at most fourth root singularities at the endpoints $a_j, b_j$. Thus $H$ has
at most square root singularities, and it follows that the singularities
are removable. Thus $H$ has an extension to an entire function.
Since $H(z) = I + O(1/z)$ as $z \to \infty$, we find by Liouville's theorem
that $H(z) = I$ for every $z \in \mathbb C$.
The proposition follows.
\end{proof}

\subsection{Zeros}

By \eqref{def:v1v2} and \eqref{eq:defM} the entries of $M(z)$ do not
vanish at any point $z \in \mathbb C \setminus [a_1,b_N]$. Across
each cut $(a_k,b_k)$, $k=1, \ldots, N$, and each gap $(b_k,
a_{k+1})$, $k=1, \ldots, N-1$ the entries have an analytic
continuation given by the RH problem. By \eqref{eq:vjatPk} we have
that $v_j^{(\nu)}(z) \to 0$ as $z \to z(P_k^{(\nu)})$ if
$P_k^{(\nu)}$ is on the $j$th sheet. By \eqref{eq:defM} this
translates into the following statement about the zeroes of the
entry $M_{\nu,j}$ of $M$.

\begin{proposition} \label{prop:Mzeros}
If $P_k^{(\nu)}$ is on the $j$th sheet then
$M_{\nu,j}(z)$ has a simple zero at $z = z(P_k^{(\nu)})$,
in the sense that the restriction of $M_{\nu,j}$
to the upper (lower) half-plane has an analytic
continuation across $(b_k,a_{k+1})$ into the lower (upper)
half-plane which has a simple zero at $z=z(P_k^{(\nu)})$.

The points $z=z(P_k^{(\nu)})$ are the only possible
zeros of the (analytic continuations of the) entries of $M$.
\end{proposition}
\begin{proof}
Everything is already proved, except for the fact that the zero is
simple. This follows from the fact that $M_{\nu,j} = \pm
e^{u_j^{(\nu)}}$ where $u_j^{(\nu)}(z)$ has the behavior
\eqref{eq:ujatPknu} as $z \to z(P_k^{(\nu)})$.
\end{proof}

\subsection{Uniform boundedness}
The solution $M$ clearly depends on $n$. The following uniform
boundedness property
is needed in the construction of a local parametrix in the
steepest descent analysis, see \cite{DKMVZ}.

\begin{proposition} \label{prop:uniformbound}
For every $\varepsilon > 0$, we have that
$M(z)$ and $M^{-1}(z)$   are uniformly bounded in $n$
for $z$ in the set
\begin{equation} \label{eq:setepsilon}
     \{ z \in \mathbb C \setminus [a_1,b_N] \mid |z-a_j| \geq \varepsilon, |z-b_j| \geq \varepsilon
    \textrm{ for all } j=1, \ldots, N \}.
    \end{equation}
\end{proposition}
\begin{proof}
It follows easily from our construction that for every $(\beta_1,
\ldots, \beta_{N-1}) \in (\mathbb R \slash \mathbb Z)^{N-1}$ there
is a unique solution to the model RH problem where the jump on
$(a_j, b_{j+1})$ is replaced by
\[ M_+(x) = M_-(x) \begin{pmatrix} e^{-2\pi i \beta_j} & 0 \\ 0 & e^{2\pi i \beta_j}
    \end{pmatrix}, \qquad x \in (a_j, b_{j+1}) \]
for $j=1, \ldots, N-1$. If we denote the solution of this RH problem
by
\begin{equation} \label{eq:Mwithbetas}
    M(z; \beta_1, \ldots, \beta_{N-1})
    \end{equation}
then the solutions we are interested in are
\[ M(z; n \alpha_1, \ldots, n \alpha_{N-1}), \qquad n \in \mathbb N, \]
and so they are part of this family \eqref{eq:Mwithbetas}. It is therefore
enough to show that the solutions \eqref{eq:Mwithbetas} are uniformly bounded
on the set \eqref{eq:setepsilon}.

The map
\[ (\beta_1, \ldots, \beta_{N-1}) \mapsto M(z; \beta_1, \ldots, \beta_{N-1}) \]
is continuous as a map from $(\mathbb R \slash \mathbb Z)^{N-1}$ to
the $2 \times 2$-matrix valued analytic functions on $\mathbb C
\setminus [a_1, b_N]$ provided with the topology of uniform
convergence on compact subsets of $\mathbb C \setminus \{a_1, b_1,
\ldots, a_N, b_N \}$. Since $(\mathbb R \slash \mathbb Z)^{N-1}$  is
compact, we then have that the functions $M(z; \beta_1, \ldots,
\beta_{N-1})$ are also compact in this space, which in particular
implies that the functions are uniformly bounded on every set of the
form \eqref{eq:setepsilon}. So $M$ is uniformly bounded in $n$ on
\eqref{eq:setepsilon}.

Since $\det M \equiv 1$, the entries of $M^{-1}$ are, up to a sign, the same as those
of $M$, and so $M^{-1}$ is also uniformly bounded in $n$ on \eqref{eq:setepsilon}.
\end{proof}

\section{Alternative construction for the second row}
\label{section5}
The above construction of the solution $M$ of the model RH problem
is done row by row. Indeed, the case $\nu = 1$ leads to
the first row, and the case $\nu=2$ leads to the second row
of $M$.
The difference between the two cases lies in the condition \eqref{eq:res3}
that specifies which point at infinity is a pole of the meromorphic differential.
Otherwise the two cases are similar, and we treated them simultaneously.

The fact that these two cases are similar is also related to the
hyperelliptic Riemann surface, which possesses hyperelliptic
involution interchanging the sheets. In other situations related to
multi-sheeted Riemann surfaces as in \cite{DuKu}, \cite{Mo}, there
is no simple symmetry between the sheets. In addition, the point at
infinity is a branch point in \cite{DuKu}, \cite{Mo} that connects
all but one of the sheets. For the construction of a model RH
problem in such situations, it may be of interest to realize that
the construction of one row of the model RH problem can help to
construct the other rows.

We illustrate this here for the hyperelliptic case, and
so we give an alternative way to construct the second row, on
the assumption that we know the first row of $M$. We thank Alexander
Aptekarev for this remark.

Recall that the construction of the first row of $M$ is based on the
points $P_k^{(1)} \in \Gamma_k$, $k=1, \ldots, N-1$   satisfying
(\ref{eq:Psiimage}). In Lemma \ref{lem:nonspecial} we proved that
the divisor
\[ D = \sum_{k=1}^{N-1} P_k^{(1)} \]
is non-special.

\begin{lemma}\label{lem:dimLD}
The vector space of meromorphic functions on $\mathcal{R}$
(including constant functions)  with divisor greater than or equal
to $-\sum_{k=1}^{N-1} P_k^{(1)} - \infty_2$ is of dimension $2$.
\end{lemma}
\begin{proof}
Let us denote, for a positive divisor $D'$, the space of meromorphic
functions on $\mathcal R$ whose divisor is greater than or equal to
$-D'$ by $L(D')$.

Since $D = \sum_{k=1}^{N-1} P_k^{(1)}$ is non-special, the only
functions in $L(D)$ are constant functions, so that
\[ \dim L(D) = 1. \]

Also, $L(D)$ is the kernel of the linear functional from $L(D +
\infty_2)$ to $\mathbb C$ that maps a function in $L(D+\infty_2)$ to
its residue at $\infty_2$. Thus by the dimension theorem for linear
functionals,
\[ \dim L(D + \infty_2) \leq \dim L(D) + 1 = 2. \]
On the other hand, the divisor $D +\infty_2$ is of degree $N$, so
that by the Riemann-Roch theorem, see \cite{FK},
\[ \dim L(D + \infty_2) \geq 2. \]
This proves the lemma.
\end{proof}

Suppose now that we have the first row of $M$.
We can then construct the second row of $M$ by modifying the first
row $(M_{11}, M_{12})$ by a meromorphic factor as follows.

By Lemma \ref{lem:dimLD}, the space of meromorphic functions $F$ on
$\mathcal R$ with divisor $\geq -\sum_{k=1}^{N-1} P_k^{(1)} -
\infty_2$ is two-dimensional. The following two conditions on $F$
(recall that $M_{12}$ is analytic in a neighborhood of $\infty_2$
with a simple zero at $\infty_2$)
\begin{align*}
    F(\infty_1)  = 0, \qquad \text{and} \qquad
    \lim_{z \to \infty_2} M_{12}(z) F(z)  = 1,
\end{align*}
determine $F$ uniquely.

We use $F_1$ and $F_2$ to denote the restrictions of $F$ to
the first and second sheet, respectively.
Then the two functions defined by
\begin{equation} \label{def:M21M22}
    M_{21}(z) = M_{11}(z) F_1(z), \qquad M_{22}(z) = M_{12}(z) F_2(z)
        \end{equation}
satisfy all conditions that we need for the entries in the second
row of $M$. Note also that the poles of $F_1$ and $F_2$ at
$P_1^{(1)}, \ldots, P_{N-1}^{(1)}, \infty_2$ are cancelled by the
zeros of $M_{11}$ and $M_{12}$ at these points, see also Proposition
\ref{prop:Mzeros}.

\section*{Acknowledgement}

We thank Alexander Aptekarev for helpful discussions.

\bibliographystyle{amsplain}

\end{document}